\documentclass[12pt,reqno]{amsart}
\usepackage{graphicx,  amssymb, fullpage, verbatim,mathrsfs, paralist, hyperref}

\theoremstyle{plain}
\newtheorem{theorem}{Theorem}[section]
\newtheorem{lemma}[theorem]{Lemma}

\theoremstyle{definition}
\newtheorem{definition}{Definition}[section]

\newtheorem{remark}[definition]{Remark}

\newcommand{\bbE}{{\ensuremath{\mathbb{E}}} }
\newcommand{\bbN}{{\ensuremath{\mathbb{N}}} }

\newcommand{\bbZ}{{\ensuremath{\mathbb{Z}}} }

\newcommand{\e}{\varepsilon}

\newcommand{\eps}{\mathcal{E}_j}
\renewcommand{\le}{\leqslant}
\renewcommand{\ge}{\geqslant}
\renewcommand{\leq}{\leqslant}

\begin{document}

\title{Improved bounds in the scaled Enflo type inequality  for Banach spaces}\thanks{O.~G. was partially supported by NSF grant CCF-0635078.  A.~N. was supported in part by NSF grants CCF-0635078 and CCF-0832795, BSF
grant 2006009, and the Packard Foundation.}
\author{Ohad Giladi}
\address{Courant Institute\\ New York University}
\email{giladi@cims.nyu.edu}
\author{Assaf Naor}
\address{Courant Institute\\ New York University}
\email{naor@cims.nyu.edu}

\begin{abstract}
It is shown that if $(X,\|\cdot\|_X)$ is a Banach space with Rademacher type $p\ge 1$ then for every $n \in \bbN$ there exists an even integer $m \lesssim n^{2-1/p}\log n$ such that for every $f:\bbZ_m^n \to X$,
\begin{align*}
\bbE_{x,\e}\Bigg[\left\|{f\left({x+\frac m 2 \e}\right)-f(x)}\right\|_X^p\Bigg] \lesssim_X m^p\sum_{j=1}^n\bbE_x\Big[\|f(x+e_j)-f(x)\|_X^p\Big],
\end{align*}
where the expectation is with respect to uniformly chosen $x \in \bbZ_m^n$ and $\e \in \{-1,1\}^n$. This improves a bound of $m \lesssim n^{3-2/p}$ that was obtained in ~\cite{MN07}. The proof is based on an augmentation of the ``smoothing and approximation'' scheme, which was implicit in~\cite{MN07}.
\end{abstract}
\maketitle

\section{Introduction}
A Banach space $(X,\|\cdot\|_X)$ is said to have Rademacher type $p\ge 1$ if there is a constant $T<\infty$ such that for every $n \in \bbN$ and for every $x_1,x_2,\dots,x_n \in X$,
\begin{align}\label{def:rad-type}
\bbE_{\e}\Bigg[\Big\|{\sum_{j=1}^n \e_jx_j}\Big\|_X^p\Bigg] \le T^p\sum_{j=1}^n\|x_j\|_X^p,
\end{align}
where the expectation in~\eqref{def:rad-type} is taken with respect to the uniform probability measure on $\{-1,1\}^n$. By considering the case of the real line, we necessarily have $p\le 2$. The smallest possible $T$ for which~\eqref{def:rad-type} holds is denoted by $T_p(X)$. The notion of Rademacher type is clearly a linear notion, as inequality~\eqref{def:rad-type} involves random linear combinations of vectors in $X$.

A Banach space $(X,\|\cdot\|_X)$ is said to be finitely representable in a Banach space $(Y,\|\cdot\|_Y)$ if there exists a constant $D<\infty$ such that for every finite dimensional subspace $E$ of $X$ there exists a subspace $F$ of $Y$ and a map $T:E\to F$ with $\|T\|\cdot\|T^{-1}\| \le D$. A classical theorem of Ribe (see~\cite{Ribe76} and also~\cite{BL00}) states that if two Banach spaces $X$ and $Y$ are uniformly homeomorphic, then $X$ is finitely representable in $Y$ and vice versa. This theorem motivated what is now known as the ``Ribe program'': finding concrete metric characterizations of local properties of Banach spaces (a property is said to be local if it depends only on finitely many vectors). 

In particular, Ribe's theorem suggests that the notion of Rademacher type has a purely metric characterization. Finding a concrete metric characterization of Rademacher type is a long standing problem that  goes back to the work of Enflo. Following Enflo, we say that a metric space $(\mathcal M,d_{\mathcal M})$ has Enflo type $p>0$ if there exists a constant $T<\infty$ such that for every $n \in \bbN$ and every $f: \{-1,1\}^n \to \mathcal M$,
\begin{multline}\label{def:enflo-type}
\bbE_{\e}\Big[d_{\mathcal M}(f(\e),f(-\e))^p\Big]
\\ \le T^p\sum_{j=1}^n \bbE_{\e}\Big[d_{\mathcal M}(f(\e_1,\dots,\e_{j-1},\e_j,\e_{j+1}\dots,\e_n),f(\e_1,\dots,\e_{j-1},-\e_j,\e_{j+1},\dots,\e_n))^p\Big].
\end{multline}

In~\cite{Enflo77}, Enflo asked whether for the class of Banach spaces, Enflo type is equivalent to Rademacher type. Clearly, Enflo type $p$ implies Rademacher $p$: simply apply inequality~\eqref{def:enflo-type} to the function $f(\e) = \sum_{j=1}^n\e_jx_j$, and inequality~\eqref{def:rad-type} is obtained. In the other direction, Pisier proved (see~\cite[Ch.\ 7]{Pisier86}) that if a Banach space has Rademacher type $p$ then it has Enflo type $p'$ for every $p'<p$. The question of whether Rademacher type $p$ implies Enflo type $p$ remains an interesting open problem. Naor and Schechtman showed~\cite{NS02} that the answer is positive for the class of UMD Banach spaces. We refer to the work of Gromov~\cite[Sec.\ 9.1]{Gro83} and Bourgain, Milman and Wolfson~\cite{BMW86} for earlier results related to the notion of non-linear type, and to the work of Ball~\cite{Ball92} for an important variant of~\eqref{def:enflo-type} known as Markov type.

Motivated by their work on metric cotype~\cite{MN08}, Mendel and Naor defined in~\cite{MN07} the notion of scaled Enflo type. A metric space $(\mathcal M,d_{\mathcal M})$ is said to have scaled Enflo type $p>0$ with constant $\theta <\infty$ if for every $n \in \bbN$ there exists an $m \in 2\bbN$ such that for every $f:\bbZ_m^n \to \mathcal M$,
\begin{align}\label{def:scaled-enflo-type}
\bbE_{x,\e}\Bigg[d_{\mathcal M}\left({f\left({x+\frac m 2 \e}\right),f(x)}\right)^p\Bigg] \le \theta^p m^p\sum_{j=1}^n \bbE_{x}\Big[d_{\mathcal M}\left({f(x+e_j),f(x)}\right)^p\Big].
\end{align}
In~\eqref{def:scaled-enflo-type} and in what follows, $\{e_j\}_{j=1}^n$ denotes the standard basis of $\bbZ_m^n$. In~\cite{MN07} the following theorem was proved:
\begin{theorem}\label{thm:equivalence}
A Banach space $X$ has Rademacher type $p\in [1,2]$ if and only if it has scaled Enflo type $p$.
\end{theorem}
The notion of scaled Enflo type thus gives a purely metric characterization of Rademacher type. While the value of $m$ is implicit in Theorem~\ref{thm:equivalence}, it does play a crucial role. Note that by choosing $m=2$ in~\eqref{def:scaled-enflo-type}, the original Enflo type inequality~\eqref{def:enflo-type} is obtained. Therefore, finding the smallest $m$ for which inequality~\eqref{def:scaled-enflo-type} holds is a question of great interest. In~\cite{MN07} it was shown that if a Banach space has Rademacher type $p\in [1,2]$ then it has scaled Enflo type $p$ with $m \lesssim n^{3-2/p}$. Motivated by the recent progress in~\cite{GMN10}, we obtain the following improved bound on $m$:
\begin{theorem}\label{main-thm-type}
Assume that a Banach space has Rademacher type $p\in [1,2]$. Then for every $n \in \bbN$, there exists $m \in 4\bbN$ with $m \lesssim n^{2-1/p}\log n$ such that for every $f: \bbZ_m^n \to X$,
\begin{align}\label{improved-type-ineq}
\bbE_{x,\e}\Bigg[\left\|{f\left({x+\frac m 2 \e}\right) - f(x)}\right\|_X^p\Bigg] \lesssim_X m^p \sum_{j=1}^n\bbE_x\Big[\left\|{f(x+e_j)-f(x)}\right\|_X^p\Big].
\end{align}
\end{theorem}
The outline of the proof of Theorem~\ref{main-thm-type} is as follows: we begin by describing a general ``smoothing and approximation'' scheme. This scheme allows us to replace $f$ in inequality~\eqref{def:scaled-enflo-type} by a ``smoothed'' version of $f$ and then use inequality~\eqref{def:rad-type}. This is discussed in detail in Section~\ref{sec:smooth-approx}. Once we have the smoothing and approximation scheme, the proof of Theorem~\ref{main-thm-type} is straightforward. This is done in Section~\ref{sec:proof-of-thm}. Nevertheless, the smoothing and approximation scheme relies on two technical lemmas. These lemmas are proved in Section~\ref{sec:proof-of-lemmas}. The logarithmic factor in Theorem~\ref{main-thm-type} appears due to an additional complication that does not arise in~\cite{GMN10} (where there is no such logarithmic term); this is overcome here via an application of Pisier's inequality~\cite{Pisier86}. Thus, in particular, due to~\cite{NS02}, for the class of UMD spaces the logarithmic factor in Theorem~\ref{main-thm-type} can be dropped.

\subsection*{Notation}
We use $\lesssim, \gtrsim$ to indicate that an inequality holds true with an implied absolute constant. Also, we use $\lesssim_X, \gtrsim_X$ to indicate that the implied constant depends on $p$ and $T_p(X)$ and $\lesssim_p, \gtrsim_p$ if the constant depends on $p$ only. Also, $\mu$ will denote the uniform probability measure on $\bbZ_m^n$ and $\tau$ will denote the unifrom probability measure on $\{-1,1\}^n$. Finally, $[n]$ will denote the set $\{1,2,\dots,n\}$.

\section{The smoothing and approximation scheme for the case of type}\label{sec:smooth-approx}
Following~\cite{GMN10}, we investigate the approach which is implicit in~\cite{MN07}. Given $f:\bbZ_m^n \to X$ and a probability measure $\nu$ on $\bbZ_m^n$, let
\begin{align*}
f*\nu(x) = \int_{\bbZ_m^n}f(x-y)d\nu(y).
\end{align*}
For a Banach space $X$ with Rademacher type $p$, suppose that we are given a probability measure $\nu$ that satisfies the following two properties:

\medskip

\noindent{\em (A) Approximation property:}
\begin{align}\label{approximation}
\int_{\bbZ_m^n}\left\|{f*\nu(x)-f(x)}\right\|_X^pd\mu(x) \lesssim_X A^p\sum_{j=1}^n\int_{\bbZ_m^n}\|f(x+e_j)-f(x)\|_X^pd\mu(x).
\end{align}

\noindent{\em (S) Smoothing property:}
\begin{multline}\label{smoothing}
\int_{\bbZ_m^n}\int_{\{-1,1\}^n}\left\|{f*\nu(x+\e)-f*\nu(x-\e)}\right\|_X^pd\tau(\e)d\mu(x) \\  \lesssim_X  S^p\sum_{j=1}^n\int_{\bbZ_m^n}\|f(x+e_j)-f(x)\|_X^pd\mu(x).
\end{multline}

The goal is to deduce inequality~\eqref{def:scaled-enflo-type} from inequality~\eqref{def:rad-type}. It is known that~\eqref{def:scaled-enflo-type} holds for 'linear' functions $x\mapsto\sum_{j=1}^nx_jv_j$ (this statement is not completely accurate but is sufficient to give an intuition). So, the idea is to first replace $f$ by a smoothed version of it which locally linear on average. The way to measure the smoothness of the function is given by the smoothing property~\eqref{smoothing}. On the other hand, the smoothed version of $f$ has to be close enough to $f$ itself, since the final goal is to prove inequality~\eqref{def:scaled-enflo-type} for $f$. This is measured by the approximation property~\eqref{approximation}.

To obtain inequality~\eqref{def:scaled-enflo-type} from~\eqref{approximation} and~\eqref{smoothing}, choose $m \in 4\bbN$ and note that by the triangle inequality and convexity,
\begin{multline}\label{use-of-convexity-smooth-approx}
\left\|{f\left({x+\frac m 2 \e}\right)-f(x)}\right\|_X^p \le 3^{p-1}\left\|{f*\nu \left({x+\frac m 2 \e}\right)-f*\nu(x)}\right\|_X^p
\\ +3^{p-1}\left\|{f*\nu \left({x+\frac m 2 \e}\right)-f\left({x+\frac m 2 \e}\right)}\right\|_X^p
+3^{p-1}\left\|{f*\nu (x)-f(x)}\right\|_X^p.
\end{multline}
Also, by the triangle inequality and H\"older's inequality (remembering that $m$ is divisible by 4),
\begin{multline}\label{use-of-holder}
\left\|{f*\nu\left({x+\frac m 2 \e}\right)-f*\nu(x)}\right\|_X^p \le \left({\sum_{t=1}^{m/4}\left\|{f*\nu(x+ 2t \e)-f*\nu(x+2(t-1)\e)}\right\|_X}\right)^p
\\ \le \left({\frac m 4}\right)^{p-1}\sum_{t=1}^{m/4}\left\|{f*\nu(x+ 2t \e)-f*\nu(x+2(t-1)\e)}\right\|_X^p.
\end{multline}
Integrating~\eqref{use-of-convexity-smooth-approx} over $\bbZ_m^n$ while using~\eqref{use-of-holder} and the translation invariance of $\mu$,
\begin{align}
\nonumber \int_{\bbZ_m^n}\left\|{f\left({x+\frac m 2 \e}\right)-f(x)}\right\|_X^pd\mu(x) & \lesssim 3^p\int_{\bbZ_m^n}\|f*\nu(x)-f(x)\|_X^pd\mu(x)
\\ & \quad + m^p \int_{\bbZ_m^n}\|f*\nu(x+\e)-f*\nu(x-\e)\|_X^pd\mu(x).
\end{align}
Using the approximation property~\eqref{approximation} and the smoothing property~\eqref{smoothing}, we get
\begin{multline}\label{almost-final}
\int_{\bbZ_m^n}\left\|{f\left({x+\frac m 2 \e}\right)-f(x)}\right\|_X^pd\mu(x)
\\ \lesssim_X (A^p+m^pS^p)\sum_{j=1}^n\int_{\bbZ_m^n}\|f(x+e_j)-f(x)\|_X^pd\mu(x).
\end{multline}
If we could find a smoothing and approximation scheme for which $S \lesssim 1$ and $A \lesssim m$, then~\eqref{almost-final} would imply
\begin{align}\label{final}
\int_{\bbZ_m^n}\left\|{f\left({x+\frac m 2 \e}\right)-f(x)}\right\|_X^pd\mu(x) \lesssim_X m^pS^p\sum_{j=1}^n\int_{\bbZ_m^n}\|f(x+e_j)-f(x)\|_X^pd\mu(x),
\end{align}
which is precisely the desired scaled Enflo type inequality~\eqref{def:scaled-enflo-type}. Thus, the goal is to come up with a smoothing and approximation scheme with $S\lesssim 1$ and $A$ as small as possible. In~\cite{MN07} it was shown that one can choose and $A\lesssim n^{3-2/p}$. Here we show that we can in fact choose $A \lesssim n^{2-1/p} \log n$.
\begin{remark}
In the context of metric cotype, it was shown in~\cite{GMN10} that the scheme of smoothing and approximation necessarily has limitations. Specifically, it was shown that using such a scheme implies that $m$ is bigger than some function of $n$. However, in the context of type, this is no longer the case. If we assume that the notions of Rademacher type and Enflo type are in fact the same for the class of Banach spaces, then the smoothing property~\eqref{smoothing} should hold for $f$ itself in which case the approximation property~\eqref{approximation} becomes trivial.
\end{remark}

\section{Proof of Theorem~\ref{main-thm-type} }\label{sec:proof-of-thm}
Let $f:\bbZ_m^n \to X$, and fix an odd integer $0<k<m/2$. We follow the notations in~\cite{GMN10} and define the following family of averaging operators.
For $f \colon \bbZ_m^n \to X$, $k<m/2$ an odd integer and $B \subseteq [n]$, let
\begin{align}\label{def:delta}
\Delta_Bf(x)\stackrel{\mathrm{def}}{=} \frac 1 {\mu(L_B)}\int_{\bbZ_m^n} f(x+y)d\mu(y),
\end{align}
where
\[ L_B\stackrel{\mathrm{def}}{=}\{y\in \bbZ_m^n:\ \forall i\notin B,\  y_i=0;  \forall i\in[n]\ y_i \text{ is even}; d_{\bbZ_m^n}(0,y)<k \}.\]
As we mentioned in Section~\ref{sec:smooth-approx}, the proof of inequality~\eqref{def:scaled-enflo-type} will follow once we have the smoothing and approximation properties. The approximation property is given by the following lemma, which was already proved in~\cite{MN07}.
\begin{lemma}[Lemma 2.2 in~\cite{MN07}]\label{lemma:approx}
For every Banach space $X$ and for every $f: \bbZ_m^n \to X$,
\begin{align*}%\label{actual-approx}
\int_{\bbZ_m^n}\left \|{\Delta_{[n]}f(x)-f(x)}\right\|_X^pd\mu(x) \le (k-1)^p n^{p-1}\sum_{j=1}^n\int_{\bbZ_m^n}\|f(x+e_j)-f(x)\|_X^pd\mu(x).
\end{align*}
\end{lemma}
The smoothing property is the new ingredient in this note.
\begin{lemma}\label{lemma:smoothing}
Assume that $X$ is a Banach space with Rademacher type $p$. Then for every $k\gtrsim n\log n$ we have for every $f \colon \bbZ_m^n \to X$,
\begin{multline*}
\int_{\bbZ_m^n}\int_{\{-1,1\}^n}\left\|{\Delta_{[n]}f(x+\e)-\Delta_{[n]}f(x-\e)}\right\|_X^pd\tau(\e)d\mu(x)
\\  \lesssim_X \sum_{j=1}^n\int_{\bbZ_m^n}\|f(x+e_j)-f(x)\|_X^pd\mu(x).
\end{multline*}
\end{lemma}
The proof of Theorem~\ref{main-thm-type} now follows immediately.
\begin{proof}[Proof of Theorem~\ref{main-thm-type}]
As we saw in Section~\ref{sec:smooth-approx}, we must have $S \lesssim 1$ and $A \lesssim m$. By Lemma~\ref{lemma:approx} and Lemma~\ref{lemma:smoothing} the smoothing and approximation properties hold with $A =  (k-1)n^{1-1/p}$ and $S=1$, assuming that $k \gtrsim n\log n$. This implies $A \gtrsim n^{2-1/p}\log n$ and therefore $m \gtrsim n^{2-1/p}\log n$.
\end{proof}

\section{Proof of Lemma~\ref{lemma:approx} and Lemma~\ref{lemma:smoothing} }\label{sec:proof-of-lemmas}
Lemma~\ref{lemma:approx} was already proved in~\cite{MN07}. However, we repeat its proof for the sake of completeness in Subsection~\ref{subsection:approx}. The proof of Lemma~\ref{lemma:smoothing} is presented in Subsection~\ref{subsection:smoothing}.

\subsection{Proof of Lemma~\ref{lemma:smoothing} }\label{subsection:smoothing}
We recall some notation from~\cite{GMN10}: for $\e \in \{-1,1\}^n$ and $B \subseteq [n]$, let $\e_B$ be the restriction of $\e$ to the coordinates of $B$. Also, for $\e,\e'\in\{-1,1\}^n$, let $\langle\e,\e'\rangle = \sum_{j=1}^n\e_j\e'_j$.
Fix $x\in\bbZ_m^n$ and $\e \in \{-1,1\}^n$. Let
\begin{align}\label{def:R}
R_{i,l}f(x,\e) \stackrel{\mathrm{def}}{=} \sum_{\substack{ S \subseteq [n] \\ |S|=i}} \sum_{\substack{\delta \in \{-1,1\}^{S} \\ \langle
\delta_{S}, \e_{S} \rangle=i-2l}}\left[{\Delta_{[n]\setminus S}f(x+\delta_Sk+\e_{[n]\setminus S}) -\Delta_{[n]\setminus S}f(x+\delta_Sk-\e_{[n]\setminus S})}\right].
\end{align}
Also, recall from~\cite{MN08} the following averaging operators:
\begin{align}\label{eq:def:Upsilon}
\eps f(x) \stackrel{\mathrm{def}}{=} \frac 1 {\mu(S(j,k))} \int _{\bbZ_m^n}f(x+y)d\mu(y),
\end{align}
where
\begin{align*}
S(j,k)&\stackrel{\mathrm{def}}{=} \{y\in\bbZ_{m}^{n}\colon\ y_{j} \text{ is even};\
\quad \forall l \in [n]-{\{j\}}, y_{l}  \text{ is odd};\  d_{\bbZ_{m}^{n}}(0,y)\leq{k}\}.
\end{align*}
The following identity follows immediately from Lemma 3.8 in~\cite{GMN10}:
\begin{align}\label{follows-from}
\sum_{j=1}^n\e_j\left[{\eps f(x+e_j)-\eps f(x-e_j)}\right] = \sum_{i=0}^n\sum_{l=0}^i h_{i,l}\frac{k^{n-i-1}}{(k+1)^{n-1}} R_{i,l}f(x,\e),
\end{align}
where $\{h_{i,l}\}_{l \le i}$ are scalars satisfying
\begin{align}\label{initial-condition}
h_{0,0}&=1,
\\
\label{bound-h}
|h_{i,l}| &\lesssim \frac {(i-l)!l!}{2^i}.
\end{align}
Note that the first term on the right hand side of~\eqref{follows-from} equals
\begin{align*}
\left({\frac k {k+1}}\right)^{n-1}R_{0,0}f(x,\e) = \left({\frac k {k+1}}\right)^{n-1}\left[{\Delta_{[n]}f(x+\e)-\Delta_{[n]}f(x-\e)}\right].
\end{align*}
Using identity~\eqref{follows-from}, the triangle inequality and the fact that $\left({\frac {k+1}k}\right)^{n-1} \lesssim 1$,
\begin{multline}\label{iden-for-delta}
 \left\|{\Delta_{[n]}f(x+\e)-\Delta_{[n]}f(x-\e)}\right\|_X \\ \lesssim \left\|{\sum_{j=1}^n\e_j\left[{\eps f(x+e_j)-\eps f(x-e_j)}\right]}\right\|_X
    + \left\|{\sum_{i=1}^n\sum_{l=0}^i \frac{h_{i,l}}{k^{i}}R_{i,l}f(x,\e)}\right\|_X.
\end{multline}
The triangle inequality implies
\begin{align}
\left\|{\sum_{i=1}^n\sum_{l=0}^i \frac{h_{i,l}}{k^{i}}R_{i,l}f(x,\e)}\right\|_X^p \le \left({\sum_{i=1}^n\sum_{l=0}^i \frac{|h_{i,l}|}{k^{i}}\left\|{R_{i,l}f(x,\e)}\right\|_X}\right)^p.
\end{align}
Convexity of the function $t\mapsto t^p$ and H\"older's inequality imply
\begin{align}\label{use-of-convexity-iterating}
\nonumber \left({\sum_{i=1}^n\sum_{l=0}^i \frac{|h_{i,l}|}{k^{i}}\left\|{R_{i,l}f(x,\e)}\right\|_X}\right)^p & =  \left({\sum_{i=1}^n2^{-(i+1)}\sum_{l=0}^i 2^{i+1}\frac{|h_{i,l}|}{k^{i}}\left\|{R_{i,l}f(x,\e)}\right\|_X}\right)^p
\\ \nonumber & \le \sum_{i=1}^n2^{-(i+1)}\left({\sum_{l=0}^i2^{i+1}\frac{|h_{i,l}|}{k^{i}}\left\|{R_{i,l}f(x,\e)}\right\|_X}\right)^p
\\  & \le \sum_{i=1}^n\sum_{l=0}^i2^{(i+1)(p-1)}(i+1)^{p-1}\frac{|h_{i,l}|^p}{k^{ip}}\left\|{R_{i,l}f(x,\e)}\right\|_X^p.
\end{align}
Therefore, combining~\eqref{iden-for-delta} and~\eqref{use-of-convexity-iterating}, we get
\begin{align}\label{bound-on-delta}
\nonumber \left\|{\Delta_{[n]}f(x+\e)-\Delta_{[n]}f(x-\e)}\right\|_X^p & \lesssim \left\|{\sum_{j=1}^n\e_j\left[{\eps f(x+e_j)-\eps f(x-e_j)}\right]}\right\|_X^p
\\ &\quad +\sum_{i=1}^n\sum_{l=0}^i2^{(i+1)(p-1)}(i+1)^{p-1}\frac{|h_{i,l}|^p}{k^{ip}}\left\|{R_{i,l}f(x,\e)}\right\|_X^p.
\end{align}
The next goal is to estimate the average of $\left\|{R_{i,l}f(x,\e)}\right\|_X^p$ over $x \in \bbZ_m^n$ and $\e \in \{-1,1\}^n$.
\begin{lemma}\label{lemma:bound-R}
Assume that $X$ is a Banach space with Rademacher type $p$. Then for all $0 \le l \le i \le n$,
\begin{multline}\label{lemma-bound-R}
\int_{\bbZ_m^n}\int_{\{-1,1\}^n}\left\|{R_{i,l}f(x,\e)}\right\|_X^pd\tau(\e)d\mu(x)
\\ \lesssim_X  (\log n)^p\binom n i ^{p} \binom i l^{p}
 \cdot \sum_{j=1}^n\int_{\bbZ_m^n}\|f(x+e_{j})-f(x)\|_X^pd\mu(x).
\end{multline}
\end{lemma}
\begin{proof}
First, note that on the right hand side of~\eqref{def:R} there are $\binom i l$ terms for which $\langle \delta_S,\e_S \rangle =i-2l$ and there are $\binom n i$ sets $S$ of size $i$. Therefore, applying the triangle inequality and then H\"older's inequality,
\begin{multline}\label{use-of-convexity-for-R}
\left\|{R_{i,l}f(x,\e)}\right\|_X^p \le \binom n i ^{p-1} \binom i l^{p-1}
\\  \cdot \sum_{\substack{ S \subseteq [n] \\ |S|=i}} \sum_{\substack{\delta \in \{-1,1\}^{S} \\ \langle
\delta_{S}, \e_{S} \rangle=i-2l}}\|\Delta_{[n]\setminus S}f(x+\delta_Sk+\e_{[n]\setminus S}) -\Delta_{[n]\setminus S}f(x+\delta_Sk-\e_{[n]\setminus S})\|_X^p.
\end{multline}
We would like to integrate inequality~\eqref{use-of-convexity-for-R} over $x\in \bbZ_m^n$ and $\e \in \{-1,1\}^n$. First, note that by the invariance of $\mu$ we have
\begin{multline}\label{use-of-invariance}
\int_{\bbZ_m^n}\left\|{\Delta_{[n]\setminus S}f(x+\delta_Sk+\e_{[n]\setminus S})- \Delta_{[n]\setminus S}f(x+\delta_Sk-\e_{[n]\setminus S})}\right\|_X^pd\mu(x)
\\ = \int_{\bbZ_m^n}\left\|{\Delta_{[n]\setminus S}f(x+\e_{[n]\setminus S})- \Delta_{[n]\setminus S}f(x-\e_{[n]\setminus S})}\right\|_X^pd\mu(x).
\end{multline}
Since $\Delta_{[n]\setminus S}$ is a convolution with a probability measure we also have
\begin{multline}\label{bound-on-delta2}
\int_{\bbZ_m^n}\left\|{\Delta_{[n]\setminus S}f(x+\e_{[n]\setminus S})- \Delta_{[n]\setminus S}f(x-\e_{[n]\setminus S})}\right\|_X^pd\mu(x)
\\ \le  \int_{\bbZ_m^n}\left\|{f(x+\e_{[n]\setminus S})- f(x-\e_{[n]\setminus S})}\right\|_X^pd\mu(x).
\end{multline}
Thus, integrating~\eqref{use-of-convexity-for-R} over $x \in \bbZ_m^n$ while using~\eqref{use-of-invariance} and~\eqref{bound-on-delta2},
\begin{multline}\label{estimate-R}
\int_{\bbZ_m^n}\left\|{R_{i,l}f(x,\e)}\right\|_X^pd\mu(x)
\\ \le  2^p\binom n i ^{p-1} \binom i l^{p}\sum_{\substack{ S \subseteq [n] \\ |S|=i}} \int_{\bbZ_m^n}\|f(x+\e_{[n]\setminus S}) -f(x-\e_{[n]\setminus S})\|_X^pd\mu(x).
\end{multline}
Recall Pisier's inequality (see~\cite[Ch.\ 7]{Pisier86}): for every $g: \{-1,1\}^n \to X$,	
\begin{multline*}%\label{pisier}
\int_{\{-1,1\}^n}\left\|{g(\e)-\int_{\{-1,1\}^n} gd\tau}\right\|_X^pd\tau(\e)
\\ \le \left({e\log n}\right)^p\int_{\{-1,1\}^n\times \{-1,1\}^n}\left\|{\sum_{j=1}^n\e_j'\left[{g\left(\e^{(j)}\right)-g(\e)}\right]}\right\|_X^pd\tau(\e')d\tau(\e),
\end{multline*}
where $\e^{(j)} \stackrel{\mathrm{def}}= (\e_1,\dots,\e_{j-1},-\e_{j},\e_{j+1},\dots,\e_n)$. For a fixed $x \in \bbZ_m^n$ and $S \subseteq [n]$, define $g_x: \{-1,1\}^n\to X$ to be
$g_x(\e)\stackrel{\mathrm{def}}{=}f(x+\e_{[n]\setminus S})-f(x-\e_{[n]\setminus S})$.
Clearly, $\int_{\{-1,1\}^n} gd\tau = 0$.
Applying Pisier's inequality to $g_x$ thus implies,
\begin{multline}\label{use-of-pisier}
\int_{\{-1,1\}^n}\left\|{g_x(\e)}\right\|_X^p d\tau(\e)
\\ \le \left({e\log n}\right)^p \int_{\{-1,1\}^n\times\{-1,1\}^n}\left\|{\sum_{j =1}^n \e'_j \left[{g_x\left(\e^{(j)}\right)-g_x(\e)}\right]}\right\|_X^pd\tau(\e')d\tau(\e).
\end{multline}
Applying the Rademacher type property of $X$, we get
\begin{align}\label{use-of-type-g}
\int_{\{-1,1\}^n}\left\|{\sum_{j =1}^n \e'_j\left[{g_x\left(\e^{(j)}\right)-g_x(\e)}\right]}\right\|_X^pd\tau(\e')
& \lesssim_X \sum_{j =1}^n\left\|{g_x\left(\e^{(j)}\right)-g_x(\e)}\right\|_X^p.
\end{align}
Now, by the definition of $g_x$ it follows immediately that $$g_x(\e)-g_x\left(\e^{(j)}\right) = \left[f(x+\e_{[n]\setminus S})-f\left(x+\e_{[n]\setminus S}^{(j)}\right)\right]+ \left[f(x-\e_{[n]\setminus S})-f\left(x-\e_{[n]\setminus S}^{(j)}\right)\right]$$ (in the case $j\in S$ we let $\e^{(j)}_{[n]\setminus S} = \e_{[n]\setminus S}$, in which case the difference is zero). Thus, using convexity and the translation invariance of $\mu$,
\begin{align}\label{bound-g-der}
\nonumber \int_{\bbZ_m^n}\left\|{g_x(\e^{(j)})-g_x(\e)}\right\|_X^pd\mu(x) & \le 2^p \int_{\bbZ_m^n}\left\|{f(x+e_j)-f(x-e_j)}\right\|_X^pd\mu(x)
\\ & \le 4^p \int_{\bbZ_m^n}\left\|{f(x+e_j)-f(x)}\right\|_X^pd\mu(x).
\end{align}
Integrating~\eqref{use-of-pisier} over $x\in \bbZ_m^n$ and using~\eqref{use-of-type-g} and~\eqref{bound-g-der}, we get
\begin{multline}\label{final-est-diagonal}
\int_{\bbZ_m^n}\int_{\{-1,1\}^n}\left\|{f(x+\e_{[n]\setminus S})-f(x-\e_{[n]\setminus S})}\right\|_X^pd\tau(\e)d\mu(x)
\\ \lesssim_X (\log n)^p \sum_{j=1}^n\int_{\bbZ_m^n}\left\|{f(x+e_j)-f(x)}\right\|_X^pd\mu(x).
\end{multline}
Plugging~\eqref{final-est-diagonal} into~\eqref{estimate-R}, we get
\begin{multline}
\int_{\bbZ_m^n}\int_{\{-1,1\}^n}\left\|{R_{i,l}f(x,\e)}\right\|_X^pd\tau(\e)d\mu(x)
\\ \nonumber  \lesssim_X  (\log n)^p\binom n i ^{p} \binom i l^{p}\sum_{j=1}^n\int_{\bbZ_m^n}\|f(x+e_j)-f(x)\|_X^pd\mu(x).
\end{multline}
The proof of Lemma~\ref{lemma:bound-R} is therefore complete.
\end{proof}
We are now in a position to prove Lemma~\ref{lemma:smoothing}.
\begin{proof}[Proof of Lemma~\ref{lemma:smoothing}]
Integrating inequality~\eqref{bound-on-delta} over $x\in \bbZ_m^n$ and $\e\in \{-1,1\}^n$, we get
\begin{multline}\label{bound-on-delta-int}
\int_{\bbZ_m^n}\int_{\{-1,1\}^n}\left\|{\Delta_{[n]}f(x+\e)-\Delta_{[n]}f(x-\e)}\right\|_X^pd\tau(\e)d\mu(x)
\\ \lesssim \int_{\bbZ_m^n}\int_{\{-1,1\}^n}\left\|{\sum_{j=1}^n\e_j\left[{\eps f(x+e_j)-\eps f(x-e_j)}\right]}\right\|_X^pd\tau(\e)d\mu(x)
\\ +\sum_{i=1}^n\sum_{l=0}^i2^{(i+1)(p-1)}(i+1)^{p-1}\frac{|h_{i,l}|^p}{k^{ip}}\int_{\bbZ_m^n}\int_{\{-1,1\}^n}\left\|{R_{i,l}f(x,\e)}\right\|_X^pd\tau(\e)d\mu(x).
\end{multline}
Applying the Rademacher type inequality to the vectors $\{\eps f(x+e_j)-\eps f(x-e_j)\}_{j=1}^n$,
\begin{multline}\label{use-of-type-eps}
\int_{\{-1,1\}^n}\left\|{\sum_{j=1}^n \e_j\left[{\eps f(x+e_j)-\eps f(x-e_j)}\right]}\right\|_X^pd\tau(\e)
\\ \lesssim_X \sum_{j=1}^n\left\|{\eps f(x+e_j)-\eps f(x-e_j)}\right\|_X^p.
\end{multline}
Integrating~\eqref{use-of-type-eps} over $x\in\bbZ_m^n$ and using convexity and the fact the $\eps$ is a convolution with a probability measure implies
\begin{multline}\label{bound-first-term}
\int_{\bbZ_m^n}\int_{\{-1,1\}^n}\left\|{\sum_{j=1}^n \e_j\left[{\eps f(x+e_j)-\eps f(x-e_j)}\right]}\right\|_X^pd\tau(\e)d\mu(x)
\\ \lesssim_X \sum_{j=1}^n\int_{\bbZ_m^n}\|f(x+e_j)-f(x)\|_X^pd\mu(x).
\end{multline}
It remains to bound the second term in~\eqref{bound-on-delta-int}. For that, fix $1 \le i \le n$ and $0 \le l \le i$. Using Lemma~\ref{lemma:bound-R} and the estimate~\eqref{bound-h}, each term in the sum on the right hand side of~\eqref{bound-on-delta-int} can be bounded as follows:
\begin{multline}\label{boundtail}
2^{(i+1)(p-1)}(i+1)^{p-1}\frac{|h_{i,l}|^p}{k^{ip}}\int_{\bbZ_m^n}\int_{\{-1,1\}^n}\left\|{R_{i,l}f(x,\e)}\right\|_X^pd\tau(\e)d\mu(x)
\\ \lesssim_X 2^{i(p-1)}(i+1)^{p-1}\left({\frac{(i-l)!l!}{2^ik^{i}}}\right)^p \binom ni^{p}\binom il^p (\log n)^p \cdot \sum_{j=1}^n\int_{\bbZ_m^n}\|f(x+e_{j})-f(x)\|_X^pd\mu(x).
\end{multline}
Now,
\begin{align*}%\label{bound-coeff}
2^{i(p-1)}(i+1)^{p-1}\left({\frac{(i-l)!l!}{2^ik^{i}}}\right)^p\binom ni^p\binom il^p (\log n)^p
&= \frac{(i+1)^{p-1}}{2^{i}}\left({\frac{n!}{(n-i)!k^i}}\right)^p(\log n)^p
\\ &\lesssim_p \frac{(i+1)^{p-1}}{2^{i}}\left({\frac{n}{k}}\right)^{ip}(\log n)^p.
\end{align*}
Thus,~\eqref{boundtail} becomes
\begin{multline}\label{bound-second-part}
2^{(i+1)(p-1)}(i+1)^{p-1}\frac{|h_{i,l}|^p}{k^{ip}}\int_{\bbZ_m^n}\int_{\{-1,1\}^n}\left\|{R_{i,l}f(x,\e)}\right\|_X^pd\tau(\e)d\mu(x)
\\ \lesssim_X \frac{(i+1)^{p-1}}{2^{i}}\left({\frac{n}{k}}\right)^{ip}(\log n)^p \cdot \sum_{j=1}^n\int_{\bbZ_m^n}\|f(x+e_{j})-f(x)\|_X^pd\mu(x).
\end{multline}
Plugging~\eqref{bound-second-part} and~\eqref{bound-first-term} into~\eqref{bound-on-delta-int}  implies
\begin{multline}\label{final-estimate-delta}
\int_{\bbZ_m^n}\int_{\{-1,1\}^n}\left\|{\Delta_{[n]}f(x+\e)-\Delta_{[n]}f(x-\e)}\right\|_X^pd\tau(\e)d\mu(x)
\\ \lesssim_X \left[1+{\sum_{i=1}^n\frac{(i+1)^{p}}{2^{i}}\left({\frac{n}{k}}\right)^{ip}(\log n)^p}\right]\sum_{j=1}^n\int_{\bbZ_m^n}\|f(x+e_{j})-f(x)\|_X^pd\mu(x).
\end{multline}
We always have $\frac{(i+1)^{p}}{2^{i}} \lesssim_p 1$. Thus, if we choose $k \gtrsim n\log n$, Lemma~\ref{lemma:smoothing} follows from~\eqref{final-estimate-delta}.
\end{proof}

\subsection{Proof of Lemma~\ref{lemma:approx} }\label{subsection:approx}
Given $z \in \bbZ_m^n$let $\|z\| \stackrel{\mathrm{def}}{=} \sum_{j=1}^n|z_j|$, where  $|z_j| \stackrel{\mathrm{def}}{=} \min\{z_j,m-z_j\}$. For $z_j \neq 0$, we let $\mathrm{sgn}(z_j)=1$ if $z_j=|z_j|$ and $\mathrm{sgn}(z_j)=-1$ otherwise.
By the triangle inequality, for every $x,z \in \bbZ_m^n$ we have
\begin{multline}\label{use-of-convexity2}
\|f(x+z)-f(x)\|_X  \le \sum_{j=1}^n\sum_{\ell=1}^{|z_j|}\left\|{f\left({x+\sum_{t=1}^{j-1}z_te_t+\ell\cdot \mathrm{sgn}(z_j)\cdot e_j}\right)}\right.
\\ \left.{-f\left({x+\sum_{t=1}^{j-1}z_te_t+(\ell-1)\cdot \mathrm{sgn}(z_j)\cdot e_j}\right)}\right\|_X.
\end{multline}
On the right hand side of~\eqref{use-of-convexity2} we have at most $\|z\|$ non-zero elements and therefore H\"older's inequality implies
\begin{multline}\label{use-of-holder2}
\|f(x+z)-f(x)\|_X^p  \le \|z\|^{p-1}\sum_{j=1}^n\sum_{\ell=1}^{|z_j|}\left\|{f\left({x+\sum_{t=1}^{j-1}z_te_t+\ell\cdot \mathrm{sgn}(z_j)\cdot e_j}\right)}\right.
\\ \left.{-f\left({x+\sum_{t=1}^{j-1}z_te_t+(\ell-1)\cdot \mathrm{sgn}(z_j)\cdot e_j}\right)}\right\|_X^p.
\end{multline}
Integrating~\eqref{use-of-holder2} over $x\in \bbZ_m^n$ and using the translation invariance of $\mu$, we get
\begin{align}\label{after-int}
\nonumber \int_{\bbZ_m^n}\|f(x+z)-f(x)\|_X^pd\mu(x)  &\le \|z\|^{p-1}\sum_{j=1}^n\sum_{\ell=1}^{|z_j|}\int_{\bbZ_m^n}\|f(x+\mathrm{sgn}(z_j)e_j)-f(x)\|_X^pd\mu(x)
\\ & = \|z\|^{p-1}\sum_{j=1}^n|z_j|\int_{\bbZ_m^n}\|f(x+e_j)-f(x)\|_X^pd\mu(x).
\end{align}
Since $k$ is odd, we have $|(-k,k)^n \cap (2\bbZ)^n|= k^n$, and thus we get by~\eqref{def:delta}
\begin{align}\label{delta-reminder}
\Delta_{[n]}f(x) = \frac 1 {k^n}\sum_{y\in (-k,k)^n \cap (2\bbZ)^n}f(x+y).
\end{align}
Using~\eqref{delta-reminder} and convexity implies
\begin{align}\label{use-of-convexity3}
\int_{\bbZ_m^n}\left \|{\Delta_{[n]}f(x)-f(x)}\right\|_X^pd\mu(x)\le \frac 1 {k^n}\sum_{z \in (-k,k)^n \cap (2\bbZ)^n}\int_{\bbZ_m^n}\|f(x+z)-f(x)\|_X^pd\mu(x).
\end{align}
Combining~\eqref{after-int} and~\eqref{use-of-convexity3} we get
\begin{align*}
\lefteqn{\int_{\bbZ_m^n}\left \|{\Delta_{[n]}f(x)-f(x)}\right\|_X^pd\mu(x)}
\\ &\qquad \le \frac 1 {k^n}\sum_{z \in (-k,k)^n \cap (2\bbZ)^n}\sum_{j=1}^n\|z\|^{p-1}|z_j|\int_{\bbZ_m^n}\|f(x+e_j)-f(x)\|_X^pd\mu(x)
\\ &\qquad \le (k-1)^pn^{p-1}\sum_{j=1}^n\int_{\bbZ_m^n}\|f(x+e_j)-f(x)\|_X^pd\mu(x).
\end{align*}
The proof is therefore complete. \qed

\bibliographystyle{abbrv}
\bibliography{enflo-type}

\end{document}